\documentclass[10pt]{article}

\usepackage[portrait,margin=2.54cm]{geometry}

\usepackage{amsfonts}
\usepackage{comment}
\usepackage{amssymb}
\usepackage{mathrsfs}
\usepackage{soul}
\usepackage{hyperref}
\usepackage{amssymb,amsthm,amsmath,amsfonts,amsbsy,latexsym}
\usepackage{graphicx}
\usepackage[numeric,initials,nobysame]{amsrefs}
\usepackage{upref,setspace}
\usepackage{enumerate}
\usepackage{paralist}
\usepackage{color}

\usepackage{mathtools}
\usepackage[inline]{enumitem}

\definecolor{expcol}{rgb}{1.0,0.5,0.5}
\definecolor{ecol}{rgb}{0.0, 0.5, 1.0}
\definecolor{ab}{rgb}{0.5, 0.0, 1.0}

\numberwithin{figure}{section}
\numberwithin{table}{section}

\numberwithin{equation}{section}

\mathtoolsset{showonlyrefs}

\newcommand{\PP}[0]{\mathbb{P}}
\newcommand{\MM}[0]{\mathbb{M}}
\newcommand{\DD}[0]{\mathbb{D}}
\newcommand{\WW}[0]{\mathbb{W}}
\newcommand{\HH}[0]{\mathbb{H}}
\newcommand{\Smb}[0]{\mathbb{S}}
\newcommand{\QQ}[0]{\mathbb{Q}}
\newcommand{\EE}[0]{\mathbb{E}}

\newcommand{\NN}{\mathbb{N}}

\newcommand{\RR}{\mathbb{R}}
\newcommand{\one}{\mathbf{1}}
\newcommand{\cle}{\mathcal{E}}
\newcommand{\clh}{\mathcal{H}}
\newcommand{\clf}{\mathcal{F}}
\newcommand{\clg}{\mathcal{G}}
\newcommand{\clp}{\mathcal{P}}
\newcommand{\clc}{\mathcal{C}}
\newcommand{\clx}{\mathcal{X}}
\newcommand{\cll}{\mathcal{L}}
\newcommand{\clr}{\mathcal{R}}

\newcommand{\cld}{\mathcal{D}}

\newcommand{\cly}{\mathcal{Y}}

\newcommand{\cls}{\mathcal{S}}

\newcommand{\clm}{\mathcal{M}}
\newcommand{\cla}{\mathcal{A}}

\newcommand{\Om}{\Omega}
\newcommand{\om}{\omega}

\newcommand{\veps}{\varepsilon}
\newcommand{\cln}{\mathcal{N}}

\newcommand{\clb}{\mathcal{B}}

\definecolor{expcol}{rgb}{1.0,0.5,0.5}
\definecolor{ecol}{rgb}{0.0, 0.0, 1.0}

\newtheorem{lemma}{Lemma}[section]

\newtheorem{theorem}[lemma]{Theorem}

\newtheorem{remark}[lemma]{Remark}

\newtheorem{definition}{Definition}[section]

\allowdisplaybreaks

\begin{document}

 \title{On Some Extensions of the Bou\'{e}-Dupuis Variational Formula}
 \author{Amarjit Budhiraja}

%
%

\maketitle

\abstract{The Bou\'{e}-Dupuis variational formula gives a representation for log Laplace transforms of bounded measurable functions of a finite dimensional Brownian motion on a compact time interval as an infimum of a suitable cost over a collection of non-anticipative control processes. This variational formula has proved to be very useful in studying a variety of large deviation problems. In this article we collect some extensions of this basic result that have appeared in disparate 
venues in studying a broad range of large deviation questions. Some of these results  can be found in a unified way in the recent book \cite{buddupbook}, while others, to date, have been scattered at various places in the literature. The latter category includes, in particular, variational representations,  when the stochastic dynamical system of interest has in addition to a driving L\'{e}vy noise, another source of randomness, e.g. due to a random initial condition; when the functionals of interest depends on infinite-length paths of a L\'{e}vy process; when the noise process is a Gaussian process with long-range dependence, e.g. a fractional Brownian motion, etc. The goal of this survey article is to present these diverse variational formulas in a systematic manner.
\\ \ \\ 

\begin{center}
  {\bf {\em In memory of Professor K.R. Parthasarathy.}}
 \end{center}
 }

\section{Introduction}
\label{sec:1}
The paper  \cite{boudup} established a variational formula for positive bounded functionals of paths of a finite dimensional Brownian motion over a compact time interval that has proved to be extremely useful in studying a wide range of large deviation problems. This result, now well nown as the Bou\'{e}-Dupuis variational formula, says the following. Fix $T \in (0,\infty)$.
Let $W = \{W(t): 0 \le t \le T\}$ be a $d$-dimensioanl standard Wiener process on a complete 
probability space $(\Om, \clf, \PP)$.
Denote by $\{\clg_t\}_{0\le t\le T}$ the $\PP$-augmentation of the filtration $\{\sigma\{W(s): 0 \le s \le t\}\}_{0\le t\le T}$. Let $\cla$ denote the class of $\clg_t$-predictable processes $v: [0,T]\times \Om \to \RR^d$ that satisfy $\int_0
^T \|v(t)\|^2 dt <\infty$ a.s.
Let $F: \clc_0^T(\RR^d) \to \RR$ be a bounded measurable map. Here and throughout, for a Polish space $\cle$,  $\clc_0^T(\cle)$ will denote the space of continuous functions from $[0,T]$ to $\cle$ equipped with the uniform topology. Then
\begin{equation}\label{rep:bodu}
-\log \EE \exp\{-F(W)\} = \inf_{v \in \cla} \EE \left[ \frac{1}{2} \int_0^{T} \|v(s)\|^2 ds + F(W + \int_0^{\cdot} v(s) ds)\right].
\end{equation}
In this we article we collect several important extensions of this result which have proved to be useful in a range of problems, including those from infinite dimensional stochastic dynamical systems, stochastic partial differential equations, jump-diffusions and other stochastic systems driven by L\'{e}vy processes, Gaussian processes with long-range dependence, mean-field interacting particle systems, stochastic particle systems on discrete lattices, random graph models, etc.
Several of these extensions can be found in the recent book\cite{buddupbook}, but there are a few others, such as those for  fractional Brownian motions, or those for functionals involving additional randomness than that from a L\'{e}vy process, or representations on an infinite time horizon, which have found use in different problems, that to-date have not been recorded in a single place. Our goal in this survey article is to introduce these various representations in a systematic way and explain the different contexts where they have been used.

The following notation and terminology will be used. 
A filtration  $\{\clf_t\}$ on a complete probability space $(\Om, \clf, \PP)$ is said to satisfy 
{\em the usual conditions} if the filtration is right continuous and $\clf_0$ contains all the $\PP$-null sets. On a filtered probability space $(\Om, \clf, \PP, \{\clf_t\})$ a $d$-dimensional Brownian motion $B$ is said to be a $\{\clf_t\}$- Brownian motion (or a $\{\clf_t\}$- Wiener process) if 
$\{B(t)-B(s): t\ge s\}$ is independent of $\clf_s$ for all $s \ge 0$.
For a random variable $X$ with values in some Polish space $\cle$, $\cll(\cle)$ will denote the probability distribution of $X$ which is an element of $\clp(\cle)$, namely the space of probability measures on $\cle$. Borel $\sigma$-field on $\cle$ will be denoted as $\clb(\cle)$. 
For $\gamma, \theta \in \clp(\cle)$, the relative entropy of $\gamma$ with respect to $\theta$, denoted as $R(\gamma \| \theta)$ is defined to be $\int_{\cle} \log \frac{d\gamma}{d\theta} d\gamma$
when $\gamma \ll \theta$ and $\infty$ otherwise.
For a sequence $\{X_n\}$ of $\cle$ valued random, the convergence of $X_n$ to an $\cle$ valued random variable $X$ in distribution will be denoted as $X_n \Rightarrow X$. Occasionally, to emphasize the dependence on the probability measure, the expected value on a probability space $(\Om, \clf, \PP)$ will be written as $\EE_{\PP}$. For a Hilbert space $(H, \langle \cdot, \cdot \rangle)$, $L^2([0,T]: H)$ will denote the Hilbert space of functions $f: [0,T] \to H$ such that $\int_0^T \|f(s)\|^2 ds <\infty$, where for $x \in H$, $\|x\|^2 = \langle x, x \rangle$. The Hilbert space $L^2([0,\infty): H)$ is defined similarly.
We will frequently use same notation for related but different objects in different sections of the article. For example the class of controls will usually be denoted as $\cla$ but their precise definition will change from section to section.

\section{General Filtrations and Infinite Dimensional Brownian Motions.}
\label{sec:HS}
Results in this section can be found in \cites{buddup3,buddupmar}. They also appear in the recent book 
\cite{buddupbook}. Although Theorem \ref{thm:simp} is not explicitly given in these references, it can be deduced easily from the results therein. We provide proof details for reader's convenience.

The representation given in \ref{rep:bodu} requires the class of controls to be adapted with respect to the (augmented) Brownian filtration. In the study of large deviation problems for mean field interacting particle systems \cites{buddupfis,budfanwu,budcon,budcon2} it is useful to have an extension of the representation which allows for a  filtration that is larger than the Brownian filtration. Such a representation is important in the proof of the lower bound in the large deviation principle for the path occupation measure associated with such interacting particle systems. It turns out that  an optimal point $\mu$ of the rate function, over a given set,
corresponds to the probability law of the state process, in a controlled stochastic system,  which   apriori is not known to be adapted to the filtration generated by the driving noise. This subtle but crucial issue requires a variational representation that allows for a larger filtration.

An extension in a different direction is motivated by systems driven by  infinite dimensional Brownian motions that are the basic models in the field of  stochastic partial differential equations. In order to study large deviation properties of such systems, it becomes important to establish a representation that is applicable to the various  models of infinite dimensional  Brownian motions, such as a Hilbert space valued Brownian motion, a cylindrical Brownian motion, space-time Brownian sheet, etc.

A representation that allows for both of these features was obtained in \cites{buddup3,buddupmar}. We only present the representation in the setting of a Hilbert space valued Brownian motion and refer the reader to \cite{buddupbook} for the other related settings.

We begin with some basic definitions. Fix $T<\infty$.
Let $(\Om, \clf, \PP)$ be a complete probability space on which is given a filtration $\{\clf_t\}_{0\le t \le T}$ satisfying the usual conditions. Let $(H, \langle \cdot, \cdot\rangle)$ be a real separable Hilbert space and let $Q$ be a strictly positive, symmetric, trace class operator on $H$. We recall the definition of a Hilbert space valued Wiener process.
\begin{definition}\label{def:hbm}
	A continuous $H$-valued stochastic process $\{W(t)\}_{0\le t \le T}$ is called a $Q$-Wiener process with respect to the filtration $\{\clf_t\}_{0\le t \le T}$ if the following hold:
	\begin{enumerate*}[label=(\roman*)]
  \item For every nonzero $h \in H$, $\langle Qh, h\rangle^{-1/2} \langle W(t), h\rangle$ is a one dimensional standard Wiener process.
  \item For every $h \in H$, $\langle W(t), h\rangle$ is an $\clf_t$-martingale.
\end{enumerate*}
\end{definition}
Define $H_0 \doteq Q^{1/2}H$. Then $H_0$ is a Hilbert space with inner product defined as
$$\langle h, k\rangle_0 \doteq \langle Q^{-1/2} h, Q^{-1/2} k\rangle,  \mbox{ for } h, k \in H_0.$$
Denote by $\tilde \cla$ the collection of all $\clf_t$-predictable $u:[0,T]\times \Om \to H_0$ that satisfy $\int_0^T \|u(t)\|_0^2 dt <\infty$. Also denote by $\clg_t$ the $\PP$-augmentation of the Brownian filtration $\sigma\{W(s): 0 \le s \le t\}$ and let $\cla$ be the subcollection of $\tilde \cla$ consisting of $\clg_t$-predictable $H_0$ valued processes. Then the following representation can be found in \cites{buddup3,buddupmar,buddupbook}. 
\begin{theorem}\label{thm:hilsp}
	Let $F:\clc_0^T(H) \to \RR$ be a bounded measurable map. Then 
	\begin{equation}\label{rep:hilsp}
-\log \EE \exp\{-F(W)\} = \inf_{v \in \clr} \EE \left[ \frac{1}{2} \int_0^{T} \|v(s)\|_0^2 ds + F(W + \int_0^{\cdot} v(s) ds)\right],
\end{equation}
where $\clr \in \{\cla, \tilde \cla\}$.
\end{theorem}
Thus the representation in particular says that the infimum on the right side taken over the larger collection $\tilde \cla$ is same as that over the subcollection $\cla$.
In fact the infimum can be taken over a much more smaller subcollection, a fact which is very useful when one would like to use weak convergence arguments in proofs of large deviation principles. To describe this subcollection, let for $M \in (0, \infty)$, define
$$S_M \doteq \left\{u \in L^2([0,T]:H_0): \int_0^T \|u(s)\|_0^2 ds \le M\right\}.$$
Denote by $\cla_{b,M}$ (resp. $\tilde \cla_{b,M}$) the subcollection of $\cla$ (resp. $\tilde \cla$) consisting of $u$ that take values in $S_M$. Let $\cla_b \doteq \cup_{M>0}\cla_{b,M}$
and $\tilde \cla_b \doteq \cup_{M>0}\tilde\cla_{b,M}$. Then \cite[Theorem 8.3]{buddupbook} shows that
\eqref{rep:hilsp} in fact holds for any $\clr \in \{\cla, \tilde \cla, \cla_b, \tilde \cla_b\}$.

One can take the infimum over even a further smaller class, namely the collection of simple processes. This fact was crucially exploited in \cite{banbudper} in the proof of the large deviation principle for a Brownian interacting particle system with local interactions. This work, in the proof of the large deviation upper bound, required certain estimates on Dirichlet forms associated with the controlled state processes, which in turn relied on the smoothness properties of the density functions of the laws of these controlled processes. The proof of these regularity properties given in \cite{banbudper} made key use of the fact that the controls are piecewise constant. Thus the result that in the variational representation the infimum can be taken over such controls played a central role in the proof.

We now present this result.
\begin{definition}
	A process $v \in \tilde \cla$ (resp. $\cla$) is called simple, if there exists a  $k \in \NN$, $N \in \NN$ and 
	$0= t_1\le t_2\le \cdots \le t_{k+1}=T$, such that 
	$$v(s,\om) \doteq \sum_{j=1}^k X_j(\om) \one_{(t_j, t_{j+1}]}(s), \; (s,\om) \in [0,\infty)\times \Om,$$
	where, for each $j = 1, \ldots , k$, $X_j$ is a real $\clf_{t_j}$ (resp. $\clg_{t_j}$) measurable random variable satisfying $|X_j|\le N$. 
	We denote the collection of all such simple processes as $\tilde \cla_s$ (resp. $\cla_s$).
\end{definition}
The following result says that the infimum in \eqref{rep:hilsp} in fact holds for any $\clr \in \{\cla, \tilde \cla, \cla_b, \tilde \cla_b, \cla_s, \tilde \cla_s\}$. Although the result follows readily from the results in \cite{buddup3}, below we give a brief argument for reader's convenience.
\begin{theorem}\label{thm:simp}
	Let $F:\clc_0^T(H) \to \RR$ be a bounded measurable map. Then 
	\begin{equation}\label{rep:hilsp2}
-\log \EE \exp\{-F(W)\} = \inf_{v \in \clr} \EE \left[ \frac{1}{2} \int_0^{T} \|v(s)\|_0^2 ds + F(W + \int_0^{\cdot} v(s) ds)\right],
\end{equation}
where $\clr \in \{\cla_s, \tilde \cla_s\}$.
\end{theorem}
\begin{proof}
	First suppose that $F:\clc_0^T(H) \to \RR$ is a continuous and bounded map. In that case, from
	\cite[Lemma 3.5]{buddup3}, we have that
	$$\inf_{v \in \cla_s} \EE \left[ \frac{1}{2} \int_0^{T} \|v(s)\|_0^2 ds + F(W + \int_0^{\cdot} v(s) ds)\right] = \inf_{v \in \tilde\cla} \EE \left[ \frac{1}{2} \int_0^{T} \|v(s)\|_0^2 ds + F(W + \int_0^{\cdot} v(s) ds)\right]$$
	and so the statement in the theorem follows from Theorem \ref{thm:hilsp}. Now consider a general real bounded measurable map $F$ on $\clc_0^T(H)$. Then there is a sequence of continuous maps $F_n: \clc_0^T(H) \to \RR$ such that $\|F_n\|_{\infty} \le \|F\|_{\infty} <\infty$ and $F_n \to F$ $\PP$-a.s.
	Fix $\veps>0$ and choose $v^* \in \tilde \cla_b$ such that
	 $$\EE \left[ \frac{1}{2} \int_0^{T} \|v^*(s)\|_0^2 ds + F(W + \int_0^{\cdot} v^*(s) ds)\right]\le \inf_{v \in \tilde\cla_b} \EE \left[ \frac{1}{2} \int_0^{T} \|v(s)\|_0^2 ds + F(W + \int_0^{\cdot} v(s) ds)\right]+ \veps.$$
	 Since $v^* \in \tilde \cla_b$, the law of $W+\int_0^{\cdot} v^*(s) ds$ is mutually absolutely continuous with respect to that of $W$ and so $F_n(W+\int_0^{\cdot} v^*(s) ds)$ converges a.s. to
	 $F(W+\int_0^{\cdot} v^*(s) ds)$. Thus 
	 \begin{align*}
	 &\lim_{n\to \infty}\EE \left[ \frac{1}{2} \int_0^{T} \|v^*(s)\|_0^2 ds + F_n(W + \int_0^{\cdot} v^*(s) ds)\right]\\ 
	 & \le \inf_{v \in \tilde\cla_b} \EE \left[ \frac{1}{2} \int_0^{T} \|v(s)\|_0^2 ds + F(W + \int_0^{\cdot} v(s) ds)\right]+ \veps.
	 \end{align*}
	 Using the first part of the proof, choose $v_n \in \cla_s$ such that
	 \begin{align*}
	 	&\EE \left[ \frac{1}{2} \int_0^{T} \|v_n(s)\|_0^2 ds + F_n(W + \int_0^{\cdot} v_n(s) ds)\right]\\ 
	 	&\le \EE\left[ \frac{1}{2} \int_0^{T} \|v^*(s)\|_0^2 ds + F_n(W + \int_0^{\cdot} v^*(s) ds)\right] + n^{-1}.
	 \end{align*}
	 From \cite[Lemma 2.5]{buddupbook} it follows that
	 $$\lim_{n\to \infty}\EE\left|F_n(W + \int_0^{\cdot} v_n(s) ds) - F(W + \int_0^{\cdot} v_n(s) ds)\right| = 0.$$
	 Combining the last three displays we have 
	 \begin{align*}
	 	&\inf_{v\in \cla_s}\EE\left[ \frac{1}{2} \int_0^{T} \|v(s)\|_0^2 ds + F(W + \int_0^{\cdot} v(s) ds)\right]\\ 
	 	&\le \lim_{n\to \infty}\EE \left[ \frac{1}{2} \int_0^{T} \|v_n(s)\|_0^2 ds + F(W + \int_0^{\cdot} v_n(s) ds)\right]\\ 
	 	&\le \lim_{n\to \infty}\EE\left[ \frac{1}{2} \int_0^{T} \|v^*(s)\|_0^2 ds + F_n(W + \int_0^{\cdot} v^*(s) ds)\right]\\ 
	 	&\le \inf_{v \in \tilde\cla_b} \EE \left[ \frac{1}{2} \int_0^{T} \|v(s)\|_0^2 ds + F(W + \int_0^{\cdot} v(s) ds)\right]+ \veps .
	 \end{align*}
	 Since $\veps>0$ is arbitrary, the result follows.
\end{proof}

\section{Functionals of Brownian Motions with Additional Randomness}
\label{sec:extran}
In some situations one is interested in studying the large deviations behavior of a stochastic dynamical system for which, in addition to a random driving noise, there is another source of randomness, coming e.g. from the initial  state of the stochastic system. Two such settings were studied in \cites{budfanwu,banbudper}. 
In order to motivate the extension considered in this section, we describe the setting in \cite{budfanwu} in some detail.

\subsection{Brownian particle systems with killing.}
 Let $\{X_i\}_{i \ge 1}$ be a sequence of i.i.d.\ exponential random variables with rate $1$ and let $\{B_i(t), t \ge 0\}_{i \ge 1}$ be independent $d$-dimensional standard Brownian motions independent of $\{X_i\}_{i \ge 1}$. 
Define for $t \ge 0$ the random sub-probability measure $\mu^n(t)$ as the solution to the following equation
\begin{equation}
	\label{eq_kill:mu^n(t)_intro}
	\mu^n(t) = \frac{1}{n} \sum_{i=1}^n \delta_{B_i(t)} \one_{\left\{ X_i > \int_0^t \langle\zeta,\,\mu^n(s)\rangle\,ds\right\}}.
\end{equation}
Here $\zeta: \RR^d \to \RR$ is a continuous function with sub-quadratic growth.
Since a.s., we can enumerate $\{X_i\}_{i=1}^n$ in a strictly increasing order, the unique solution of \eqref{eq_kill:mu^n(t)_intro} can be written explicitly in a recursive manner.
It can be checked (see \cite[Theorem 2.1]{budfanwu}) that $\mu^n \doteq \{\mu^n(t)\}_{t\in [0,T]}$ converges, in the Skorokhod path space $\cld \doteq \DD([0,T]: \clm(\RR^d))$, where $\clm(\RR^d)$ is the space of sub-probability measures on $\RR^d$ equipped with the weak convergence topology, in probability
to $\mu$ where for $t>0$, $\mu(t)$ has density $u(t, \cdot)$ given as the solution of 
\begin{equation}
	\label{eq_kill:PDE}
	\frac{\partial u(t,x)}{\partial t}=\frac{1}{2}\Delta u(t,x) - \langle\zeta,\,u\rangle\,u(t,x),\; (t,x) \in (0,\infty)\times \RR^d,
	\; \lim_{t\downarrow 0}u(t,\cdot) = \delta_0(\cdot),
\end{equation}
Such particle systems are motivated by problems in biology, ecology, chemical kinetics, etc. For example, the simplest case 
where $\zeta \equiv 1$ corresponds to the case where the killing rate is proportional to the total number of particles alive and models a setting in which particles compete for a common resource. More general functions $\zeta$
are of interest as well and one interpretation of $\zeta(x)$ is the amount of resource consumed by  a particle in state $x$.
Similar particle systems arise in problems of mathematical finance as models for self exciting correlated defaults. 

In studying large deviation properties of the sequence $\{\mu^n\}$ of $\cld$ valued random variables one needs to understand the asymptotics of Laplace functionals of the form
$$-\frac{1}{n} \log \EE[\exp(-nf(\mu^n))]$$
for real continuous and bounded functions $f$ on $\cld$.
In view of the weak convergence approach to the study of large deviation problems\cites{dupell4,buddupbook}, it is then natural  to derive a suitable controlled representation for the above Laplace functional and understand tightness and weak convergence properties of the various terms in this representation. Note that $\mu^n$ can be viewed as a measurable functional of the $nd$-dimensional Brownian motion $\mathbf{B}^n = (B_1, \ldots B_n)$ and an independent $\RR_+^n$ dimensional random variable 
$\mathbf{X}^n= (X_1, \ldots, X_n)$ and so one can write 
$$-\frac{1}{n} \log \EE[\exp(-nF_n(\mathbf{B}^n, \mathbf{X}^n))]$$
where $F_n$ is a suitable real bounded and measurable map on $\clc_0^T(\RR^{nd})\times \RR_+^n$.
In the case where $F_n$ only depends on $\mathbf{B}^n$, we saw that useful variational representations are given by \eqref{rep:bodu} and Theorem \ref{thm:hilsp}.
However the setting where the function $F_n$ also depends on $\mathbf{X}^n$ is not covered by the above results.

For such settings one needs a variational representation that allows for functionals that, in addition to depending on a Brownian motion, depend also on an additional source of randomness.  Such a representation was given in \cite[Proposition 4.1]{budfanwu}. We present this result below. 

\subsection{A Variational Representation with Additional Randomness.}
Let $T<\infty$, and $(\Omega,\clf,\{\clf_t\}_{0\le t \le T},\PP)$ be a  filtered probability space as in the previous section on which  are given a $d$-dimensional standard $\clf_t$-Brownian motion $W$ and an $\clf_0$-measurable random variable $X$, which takes values in a Polish space $\Smb$ and has probability law $\rho$.
We present below a  variational representation for $-\log \EE \left[\exp \left(-F(W,X)\right) \right]$, where $F$ is a real bounded measurable map on $\clc_0^T(\RR^d) \times \Smb$.
Note that since $X$ is $\clf_0$-measurable, $W$ and $X$ are independent.

Consider the probability space $(\bar\Omega,\bar\clf,\bar\PP)$ on which we are given 
a $d$-dimensional standard Brownian motion $\bar W$ and an $\Smb$-valued random variable $\bar X$, which is independent of $\bar W$, with law $\Pi$.
Let $\{\hat\clf_t\}$ be  any filtration on $(\bar\Omega,\bar\clf,\bar\PP)$ satisfying the usual conditions, such that $\bar W$ is still a standard $d$-dimensional Brownian motion with respect to $\{\hat\clf_t\}$ and $\bar X$ is 
$\hat \clf_0$-measurable.
One example of such a filtration is $\{\bar \clf_t \doteq \sigma\{\bar \clf_t^{\bar W,\bar X} \cup 
\bar\cln\}\}$, where $\bar\cln$ is the collection of all $\bar \PP$-null sets and $\bar \clf_t^{\bar W,\bar X} \doteq \sigma\{\bar X, \bar W(s):s \le t\}$.
Let $\Upsilon_{\Pi} \doteq (\bar \Omega,\bar \clf,\{\hat\clf_t\},\bar\PP)$ and consider the following collection of processes
\begin{align*}
	\cla(\Upsilon_{\Pi}) \doteq \left\{ u : \mbox{the process } u \mbox{ is }  \hat\clf_t \mbox{-predictable }  \mbox{and } \bar \EE \int_0^T \|u(s)\|^2 \,ds < \infty\right\}.
\end{align*}
The following is \cite[Proposition 4.1]{budfanwu}.
\begin{theorem}
	\label{prop_kill:rep_general}
	Let $F$ be a real bounded measurable map on $\clc_0^T(\RR^d) \times \Smb$. Then
	\begin{align*} 
		& -\log \EE \left[\exp \left(-F(W,X)\right) \right]  \\
		& \quad = \inf_{\Pi, \Upsilon_{\Pi}}\inf_{u \in \cla(\Upsilon_{\Pi})} \left\{ R(\Pi\|\rho) 
		+ \bar \EE \left[ \frac{1}{2} \int_0^T \|u(s)\|^2 \, ds + F\left(\bar W + \int_0^\cdot u(s) \, ds, 
		\bar X \right) \right] \right\}, 
	\end{align*}
	where the outer infimum is over all $\Pi \in \clp(\Smb)$ and all systems $\Upsilon_{\Pi}$. 
\end{theorem}
\begin{remark}
	Denote for $M \in (0,\infty)$, by $\cla_{b,M}(\Upsilon_{\Pi})$ the subcollection of
	$\cla(\Upsilon_{\Pi})$ consisting of controls $u$ that take value in $S_M$. Let 
	$\cla_{b}(\Upsilon_{\Pi}) \doteq \cup_{M>0}\cla_{b,M}(\Upsilon_{\Pi})$. Also, let
$\cla_s(\Upsilon_{\Pi})$ be the subcollection of
	$\cla(\Upsilon_{\Pi})$ consisting of $\hat \clf_t$-adapted simple processes. Then the collection
	$\cla(\Upsilon_{\Pi})$ in the statement of Theorem \ref{prop_kill:rep_general} can be replaced by either $\cla_{b}(\Upsilon_{\Pi})$ or $\cla_s(\Upsilon_{\Pi})$.

	Furthermore, a similar representation can be written for an infinite dimensional Brownian motion, e.g. for a $H$-valued Wiener process as in Section \ref{sec:HS}.
\end{remark}

For details of the proof of Theorem \ref{prop_kill:rep_general} we refer the reader to \cite{budfanwu} but we make some comments on the proof idea.

Consider the probability space $(\tilde \Om, \tilde \clf, \tilde \PP)$, where $\tilde\Omega = \clc_0^T(\RR^d)$, $\tilde \clf = \clb(\clc_0^T(\RR^d))$ and $\tilde\PP$ is the $d$-dimensional Wiener measure.
Namely, under $\tilde\PP$ the canonical coordinate process
$\tilde W \doteq \{\tilde W(t,\tilde\omega) \doteq \tilde\omega(t), 0 \le t \le T\}$ is a standard $d$-dimensional Brownian motion with respect to the filtration
$\{\tilde\clf_t^{\tilde W}
\doteq \{\sigma\{\tilde W(s):s \le t\}\}$.
Let $\{\tilde\clf_t\}$ be the augmented filtration, namely $\tilde\clf_t \doteq \sigma\{\tilde
\clf_t^{\tilde W} \cup \tilde\cln\}$ and $\tilde\cln$ is the collection of all $\tilde\PP$-null sets.
Define
\begin{equation}
	\label{eq_kill:ftil(xbd)_def}
	\tilde F(x) \doteq - \log \tilde\EE \left[ \exp \left( -F(\tilde W, x) \right) \right], \quad x \in \Smb.
\end{equation}
From the independence between $W$ and $X$ we see that
\begin{equation*}
	-\log \EE \left[\exp \left(-F(B,X)\right) \right] = -\log \EE\left[ \exp \left(-\tilde F(X)\right) \right].
\end{equation*}
Applying classical results of Donsker--Varadhan (cf.~\cite[Proposition 2.2]{buddupbook} ) to the right side, we have the following representation formula from the above equality
\begin{equation}
	\label{eq_kill:rep_complex}
	-\log \EE \left[\exp \left(-F(B,X)\right) \right] = \inf_{\Pi \in \clp(\Smb)} \left[ R(\Pi\|\rho) + \int_\Smb \tilde F(x) \Pi(dx) \right].
\end{equation}
Consider the collection of processes
\begin{equation*}
	\cla \doteq \left\{ u : \mbox{the process } u \mbox{ is } \tilde\clf_t \mbox{-predictable and } \tilde \EE \int_0^T \|u(s)\|^2 \,ds < \infty\right\}.
\end{equation*}
From \eqref{rep:bodu} we now have the following variational formula
\begin{equation}
	\label{eq_kill:ftil(xbd)}
	\tilde F(x) = \inf_{u \in \cla} \tilde \EE \left[ \frac{1}{2} \int_0^T \|u(s)\|^2 \, ds + F\left(\tilde W + \int_0^\cdot u(s) \, ds, x \right) \right],
\end{equation}
which together with \eqref{eq_kill:rep_complex} gives
\begin{align*}
	& -\log \EE \left[\exp \left(-F(B,X)\right) \right]  \\
	& \quad = \inf_{\Pi \in \clp(\Smb)} \left\{ R(\Pi\|\rho) 
	 + \int_\Smb \inf_{u\in \cla} \tilde\EE \left[ \frac{1}{2} \int_0^T \|u(s)\|^2 \, ds + F\left(\tilde W + \int_0^\cdot u(s) \, ds, x \right) \right] \Pi(dx) \right\}.
\end{align*}
The above representation is inconvenient to use directly in  large deviation proofs and the main challenge in using it, in comparison to the representation in Theorem \ref{prop_kill:rep_general} is the presence of the infimum under the integral on the right side. Nevertheless, the above representation is the starting point of the proof of Theorem \ref{prop_kill:rep_general} which uses discretization and measurable selection arguments to show not only that  the integral and the infimum can be interchanged in the representation but also that the minimal filtration $\tilde \clf_t$ can be replaced by a larger filtration with respect to which which $B$ is a martingale. The fact that one can allow for a general filtration turns out to be crucial in the proof of the lower bound of the large deviation proof for the collection $\{\mu^n\}$ described earlier in the section.

\section{Infinite Time Horizon }
\label{sec:infpath}
In some problems one needs a variational formula for functionals of the paths of a Brownian motion over an infinite time horizon. We describe one such situation below.
\subsection{Empirical Measures of State Processes Driven by a Fractional Brownian Motion.}
One  situation where the need to consider infinite-length  paths of Brownian motion  arises naturally is in the study of path empirical measures of ergodic solutions of stochastic differential equations driven by a fractional Brownian motion (fBM). The term ergodicity here needs to be interpreted suitably as the solution processes are not Markov, however there is a natural way to augment the state of the system with an infinite length path of a Brownian motion to get a Markovian state descriptor  (see \cite{hairFBM}). The starting point of obtaining such a Markovian description is the Mandelbrot - Van Ness decomposition of a fBM in terms of a two sided Brownian motion. This decomposition is as follows.

Let $(\Om, \clf, \PP)$ be a probability space on which is given a real two-sided Wiener process $\{W(t): -\infty < t <\infty\}$. Such a process is characterized by the property that for every $t \in \RR$, $\{W^t(s) \doteq W(t-s) - W(t)\}_{s \ge 0}$ is a standard Brownian motion independent of $\{W(t+u): u \ge 0\}$. Fix $H \in (0,1)$. Then a real fractional Brownian motion with Hurst parameter $H$, denoted as $B_H(\cdot)$, has the following representation: For $t\ge 0$,
\begin{align*}
	B_H(t) = \int_{-\infty}^0 \clg(r) (dW(r+t) - dW(r)) = \int_{-\infty}^0 (\clg(r+t) -\clg(r)) dW(r) + \int_0^t \clg(r-t) dW(r), 
\end{align*} 
where $\clg(r) = \alpha_H r^{H-1/2}$, for $r\ge 0$, and $\alpha_H$ is a suitable constant.

One is interested in stochastic differential equations (SDE) driven by such a driving noise. For simplicity of presentation, considered the simplest such equation which describes a fractional Ornstein-Uhlenbeck process:
$$dX(t) = - \alpha X(t) dt + dB_H(t).$$
The ergodicity behavior of $\{X(t), t\ge 0\}$ and in fact of much more general  processes given as solutions of  SDE driven by a fBM have been studied in \cite{hairFBM} and in several subsequent works. 
These results, in particular, characterize the law of the limit of the occupation measures $\{L^T, T >0\}$ as $T \to \infty$, where
$$L^T(\cdot) \doteq \frac{1}{T} \int_0^T \delta_{X(s)} \, ds, \; T>0.$$
In order to study the large deviation behavior of the above collection of $\clp(\RR)$ valued random variables, one needs to characterize the asymptotic behavior of Laplace functionals of the following form:
$$-\frac{1}{T} \log \EE \exp\{-T f(L^T)\}$$
where $f: \clp(\RR) \to \RR$ is a bounded continuous map. Note that $L^T$ is a measurable map of the infinite path of a standard Brownian motion, namely the collection
$\{W^T(s) = W(T-s) -W(T): s \ge 0\}$. Thus, for a suitable bounded measurable map $F_T: \clc_{0}^{\infty}(\RR) \to \RR$,
$$-\frac{1}{T} \log \EE \exp\{-T f(L^T)\} = -\frac{1}{T} \log \EE \exp\{-T F_T(W^T)\}.$$
Here, for a Polish space $\cle$, $\clc_{0}^{\infty}(\cle)$ denotes the space of continuous functions from $[0,\infty)$ to $\cle$, equipped with the  local uniform topology.
Thus, in order to study asymptotic behavior of such a quantity, it is of interest to develop a variational representation for functionals of infinite-length paths of a standard Brownian motion.

\subsection{Representation for Infinte Length Brownian Paths.}
We will only provide detailed arguments for a one dimensional Brownian motion. Analogous results can be established for a Brownian motion with values in a separable Hilbert space by combining  the arguments here with those in  \cite{buddup3}. We also remark that, in order to study ergodicity properties of $X(\cdot)$ and other fractional diffusions, \cite{hairFBM} and subsequent papers view $W^T$ as an element of a suitable path space $\clh_H$ of H\"{o}lder continuous functions. Since the Borel $\sigma$-fields on $\clh_H$ and $\clc_{0}^{\infty}(\RR) $ are the same, one immediately obtains from Theorem \ref{thm:infhor} below  also a representation for real bounded measurable functionals on $\clh_H$.

Let  $(\Om, \clf, \PP, \{\clf_t\}_{t\ge 0})$ be a filtered probability space where the filtration satisfies the usual conditions.
Let $\{W(t)\}_{t\ge 0}$ be a standard $\clf_t$-Wiener process on this space. The $\PP$-augmentation of the filtration $\{\sigma\{W(s): 0 \le s \le t\}\}_{t\ge 0}$ will be denoted as $\{\clg_t\}_{t\ge 0}$. Note that $\clg_t \subset \clf_t$ for all $t\ge 0$. We denote by $\cla$ (resp. $\tilde \cla$) the collection of all $\RR$-valued $\clg_t$-predictable (resp. $\clf_t$-predictable) processes $v: [0,\infty) \times \Om \to \RR$ that satisfy
$$\PP\left\{ \int_0^{\infty} |v(t)|^2 dt <\infty\right \} = 1.$$
It will be convenient to consider the following subcollections of $\cla$ and $\tilde \cla$ as well.
Let, for $M \in (0,\infty)$,
$$S_M \doteq \left\{u \in L^2([0,\infty):\RR): \int_0^{\infty} |u(s)|^2 ds \le M\right \},$$
and $\cla_{b,M}$ (resp. $\tilde \cla_{b,M}$) be elements of $\cla$ (resp. $\tilde \cla$) that take values in $S_M$ a.s.
Define 
$$\cla_b \doteq \cup_{M \in \NN} \cla_{b,M}, \; \tilde\cla_b \doteq \cup_{M \in \NN} \tilde\cla_{b,M}.$$
The following is the main variational representation of the section.
\begin{theorem}\label{thm:infhor}
	Let $F: \clc_0^{\infty}(\RR) \to \RR$ be a bounded and measurable map. Then
	\begin{equation}
		-\log \EE \exp\{-F(W)\} = \inf_{v \in \clr} \EE \left[ \frac{1}{2} \int_0^{\infty} |v(s)|^2 ds + F(W + \int_0^{\cdot} v(s) ds)\right],
	\end{equation}
	where $\clr\in \{\cla, \tilde \cla, \cla_b, \tilde \cla_b\}$.
\end{theorem}
In preparation for the proof of the theorem we give sime preliminary results.
For $v \in \tilde \cla_b$, and $T \in [0,\infty)$, let
$$L_{\infty}(v) \doteq \exp \left\{ \int_0^{\infty} v(s) dW(s) - \frac{1}{2} \int_0^{\infty} |v(s)|^2 ds\right\}, \; 
L_{T}(v) \doteq \exp \left\{ \int_0^{T} v(s) dW(s) - \frac{1}{2} \int_0^{T} |v(s)|^2 ds\right\}.$$
The following result is a simple consequence of Girsanov's theorem.
\begin{lemma}\label{lem:gir}
	Let $v \in \tilde \cla_b$. Then $\EE L_{\infty}(v) = 1$. Define the probability measure $\QQ$ on $(\Om, \clf)$ as
	$$\frac{d\QQ}{d\PP} \doteq L_{\infty}(v) = \exp \left\{ \int_0^{\infty} v(s) dW(s) - \frac{1}{2} \int_0^{\infty} |v(s)|^2 ds\right\}.$$
	Then 
	$$\tilde W(t) \doteq W(t) - \int_0^t v(s) ds, \; t \ge 0$$
	is a $\clf_t$-Wiener process on $(\Om, \clf, \QQ, \{\clf_t\}_{t\ge 0})$.

\end{lemma}
\begin{proof}
	Fix $v \in \tilde \cla_b$. Then there is a $M <\infty$ such that $v \in \tilde \cla_{b,M}$. Since, for every $T\in [0,\infty)$, $\int_0^T |v(s)|^2 ds \le \int_0^{\infty} |v(s)|^2 ds \le M$, we have from Novikov's criterion (cf. \cite[Corollary VIII.1.16]{revyor}) $\EE(L_T(v))= 1$ for all $T \in [0, \infty)$.
	Also, using Cauchy-Schwarz inequality, it is easy to verify that
	$\sup_{T\in [0,\infty)} \EE(L_T(v))^2 \le e^M <\infty$ which shows that $\{L_T(v), T\ge 0\}$ is a uniformly integrable $\clf_t$-martingale on $(\Om, \clf, \PP)$. Noting that
	$L_T(v) \to L_{\infty}(v)$ in probability, we now see that in fact the convergence holds in $L^1$ as well and consequently $\EE L_{\infty}(v) = 1$ proving the first statement in the lemma. Also, it follows that, for $T \in [0,\infty)$, $E(L_{\infty}(v) \mid \clf_T) = L_T(v)$ a.s. By Girsanov's theorem (cf. \cite[Theorem 3.5.1]{karshr}) it now follows that 
	$\tilde W$ is a $\clf_t$-Wiener process on $(\Om, \clf, \QQ, \{\clf_t\}_{t\ge 0})$. This completes the proof of the second statement.
\end{proof}

The following collection of simple processes will be useful.
\begin{definition}
	A process $v \in \tilde \cla$ (resp. $\cla$) is called simple, if there exists a $T \in (0,\infty)$, $k \in \NN$, $N \in \NN$ and 
	$0= t_1\le t_2\le \cdots \le t_{k+1}=T$, such that 
	$$v(s,\om) \doteq \sum_{j=1}^k X_j(\om) \one_{(t_j, t_{j+1}]}(s), \; (s,\om) \in [0,\infty)\times \Om,$$
	where, for each $j = 1, \ldots , k$, $X_j$ is a real $\clf_{t_j}$ (resp. $\clg_{t_j}$) measurable random variable satisfying $|X_j|\le N$. 
	We denote the collection of all such simple processes as $\tilde \cla_s$ (resp. $\cla_s$).
\end{definition}
Note that $\cla_s \subset \cla_b$, $\tilde \cla_s \subset \tilde \cla_b$.

Since a simple process is zero after some finite time horizon, the known results for finite time horizon filtrations (see e.g. \cite[Lemma 8.7 and Equation (8.14)]{buddupbook}) give the following result.
For $v \in \cla_b$, denote $W^v(\cdot) \doteq W(\cdot) - \int_0^{\cdot} v(s) ds$ and define the probability measure $\QQ^v$ on $(\Om, \clf)$ as
$d\QQ^v \doteq L_{\infty}(v)d\PP$.
\begin{lemma}\label{lem:simpequ}
	For every $v_0 \in \tilde \cla_s$, a bounded measurable map $F: \clc_0^{\infty}(\RR) \to \RR$, and $\veps>0$ there is a $v \in \cla_s$ such that
	\begin{align*}
		\EE\left[ \frac{1}{2} \int_0^{\infty} |v(s)|^2 ds + F(W + \int_0^{\cdot} v(s) ds)\right] \le 
		\EE\left[ \frac{1}{2} \int_0^{\infty} |v_0(s)|^2 ds + F(W + \int_0^{\cdot} v_0(s) ds)\right] + \veps.
		\end{align*}
		Furthermore, for every $v \in \cla_s$ there is a $\tilde v \in \cla_s$ such that $\QQ^{\tilde v} \circ (W^{\tilde v}, \tilde v)^{-1} = \PP \circ (W,  v)^{-1} $.
\end{lemma}
In the above lemma when talking about the distribiution of $v, \tilde v$, these are regarded as random variables in $S_M$, for a suitable $M \in (0,\infty)$, where this space is equipped with the weak toplogy inherited from that on the Hilbert space $L^2([0, \infty):\RR)$. 

We now complete the proof of Theorem \ref{thm:infhor}.\\

\noindent
{\bf Proof of Theorem \ref{thm:infhor}.}
We begin by proving the upper bound
\begin{equation}\label{eq:uppbd}
	-\log \EE \exp\{-F(W)\} \le \inf_{v \in \clr} \EE \left[ \frac{1}{2} \int_0^{\infty} |v(s)|^2 ds + F(W + \int_0^{\cdot} v(s) ds)\right]
\end{equation}
Since $\tilde \cla$ is the largest collection, it suffices to show the result with $\clr = \tilde \cla$.
We will first show the inequality in \eqref{eq:uppbd} (without the infimum) for an arbitrary $v \in \tilde \cla_s$. Note that for this, in view of the first part of Lemma \ref{lem:simpequ}, it suffices to  establish the inequality for an arbitrary $v \in \cla_s$. Consider now such a $v$.
From the second part of Lemma \ref{lem:simpequ} it follows that there is a $\tilde v \in \cla_s$ such that $\QQ^{\tilde v} \circ (W^{\tilde v}, \tilde v)^{-1} = \PP \circ (W,  v)^{-1} $.
From the Donsker-Varadhan variational formula (cf. \cite[Proposition 2.2(a)]{buddupbook}) we have that
$$-\log \EE \exp\{-F(W)\} \le R(\QQ^{\tilde v}\| \PP) + \int_{\Om} F(W) d\QQ^{\tilde v}.$$
Note that
\begin{multline*}
	R(\QQ^{\tilde v}\| \PP) = \int_{\Om} \log \left(\frac{d\QQ^{\tilde v}}{d\PP}\right) d\QQ^{\tilde v} = \EE_{\QQ^{\tilde v}}\left (\int_0^{\infty} \tilde v(s) dW(s) - \frac{1}{2} \int_0^{\infty} |\tilde v(s)|^2 ds\right)\\
	 = \EE_{\QQ^{\tilde v}}\left (\int_0^{\infty} \tilde v(s) dW^{\tilde v}(s) + \frac{1}{2} \int_0^{\infty} |\tilde v(s)|^2 ds\right) = 
	 \EE\left (\int_0^{\infty}  v(s) dW(s) + \frac{1}{2} \int_0^{\infty} |v(s)|^2 ds\right) \\
	 = \frac{1}{2} \EE\int_0^{\infty} |v(s)|^2 ds,
	\end{multline*}
where the fourth equality uses the second part of Lemma \ref{lem:simpequ}.
Also, by another application of this result, we have
$$\int_{\Om} F(W) d\QQ^{\tilde v} = \int_{\Om} F(W^{\tilde v} + \int_0^{\cdot} \tilde v(s) ds) d\QQ^{\tilde v} = \EE F(W + \int_0^{\cdot}  v(s) ds).$$
Combining the last two displays, we now have the inequality in \eqref{eq:uppbd} (without the infimum) for any $v \in \cla_s$ and thus, as discussed previously,  also for any $v \in \tilde \cla_s$.
Now consider a $v \in \tilde \cla$. Without loss of generality, $\EE\int_0^{\infty} |v(s)|^2 ds < \infty$. Thus, for each $n \in \NN$ there is a $T_n<\infty$ such that $\EE\int_{T_n}^{\infty}|v(s)|^2 ds \le n^{-1}$. Also, we can find (cf. \cite[Lemma II.1.1]{ikewat}]), for each $n\in \NN$, a  $v_n \in \tilde \cla_s$ that satisfying $v_n(t)=0$ for all $t\ge T_n$ such that 
\begin{equation}\label{eq:746}
\EE\int_0^{\infty} |v_n(t)-v(t)|^2 dt = \EE\int_0^{T_n} |v_n(t)-v(t)|^2 dt + \EE\int_{T_n}^{\infty} |v(t)|^2 dt \le 2n^{-1}.\end{equation}
Note that 
$$\sup_{n}\int_0^{\infty} |v_n(t)|^2 dt \le 2 \int_0^{\infty} |v(t)|^2 dt + 4 <\infty.$$
From the fact that the inequality in \eqref{eq:uppbd} holds for any $v \in \cla_s$, we have
\begin{equation}\label{eq:uppbd2}
	-\log \EE \exp\{-F(W)\} \le \EE \left[ \frac{1}{2} \int_0^{\infty} |v_n(s)|^2 ds + F(W + \int_0^{\cdot} v_n(s) ds)\right].
\end{equation}
Note that, from \eqref{eq:746}, for any $T < \infty$, 
$$E\sup_{0\le t \le T} |\int_0^t v_n(s) ds - \int_0^t v(s) ds|^2 \le  T E\int_0^{\infty}|v_n(s)-v(s)|^2 ds \to 0, \; \mbox{ as } n \to \infty.$$
Thus we have that $\theta_n \doteq \cll(W + \int_0^{\cdot} v_n(s) ds) \to \cll(W + \int_0^{\cdot} v(s) ds) \doteq \theta$,  as $n\to \infty$, as probability measures on $\clc_0^{\infty}(\RR)$.
Also, denoting the probability law of $W$ by $\theta_0$,   note that $\theta_0 = \QQ \circ (W+ \int_0^{\cdot} v_n(s) ds)$, where $\QQ$ is as in Lemma \ref{lem:gir} with $v$ there replaced by $-v_n$. Thus
\begin{multline*}
	R(\theta_n \| \theta_0) = R(\PP \circ (W + \int_0^{\cdot} v_n(s) ds)^{-1} \|  \QQ \circ (W + \int_0^{\cdot} v_n(s) ds)^{-1}) \\
	\le R(\PP\|\QQ) = \EE\left(\int_0^{\infty} v_n(s) dW(s) + \frac{1}{2} \int_0^{\infty} |v_n(s)|^2 ds\right) = \EE\int_0^{\infty} |v_n(s)|^2 ds.\end{multline*}
Thus $\sup_n R(\theta_n \| \theta_0) <\infty$. This together with $\theta_n \to \theta$ now implies (see \cite[Lemma 2.5]{buddupbook}) that
$F(W + \int_0^{\cdot} v_n(s) ds) \Rightarrow F(W + \int_0^{\cdot} v(s) ds)$. Together with the fact
from \eqref{eq:746} that $\EE \int_0^{\infty} |v_n(t))|^2 dt$ converges to $\EE \int_0^{\infty} |v(t))|^2 dt$, we now have from \eqref{eq:uppbd2} that this inequality holds with $v_n$ replaced by $v$. Since $v \in \cla$ is arbitrary, this proves the upper bound in \eqref{eq:uppbd} for $\clr = \tilde \cla$ and therefore for any $\clr \in \{\cla, \cla_b, \tilde \cla, \tilde \cla_b\}$.\\

\noindent We now establish the complementary lower bound:
\begin{equation}\label{eq:lowbd}
	-\log \EE \exp\{-F(W)\} \ge \inf_{v \in \clr} \EE \left[ \frac{1}{2} \int_0^{\infty} |v(s)|^2 ds + F(W + \int_0^{\cdot} v(s) ds)\right]
\end{equation}
It suffices to show the inequality for $\clr = \cla_b$ since that is the smallest collection.
Consider first the case where for some $T \in(0, \infty)$, and a bounded measurable map 
$F_T: \clc_0^T(\RR) \to \RR$, $F(w(\cdot)) = F_T(w(\cdot \wedge T))$ for all $w \in \clc_0^{\infty}(\RR)$.
In this case, denoting the subcollection of $\cla_b$ that consists of predictable $v:[0,\infty) \times \Om \to \RR$ that satisfy $v(t, \om)=0$ for $t>T$ by $\cla_{b,T}$, we have (cf. \cite[Section 8.1.4]{buddupbook}) that
\begin{align*}
	-\log \EE \exp\{-F(W)\} &\ge \inf_{v \in \cla_{b,T}} \EE \left[ \frac{1}{2} \int_0^{T} |v(s)|^2 ds + F_T\left(W(\cdot \wedge T) + \int_0^{\cdot \wedge T} v(s) ds\right)\right]\\
	&\ge \inf_{v \in \cla_{b}} \EE \left[ \frac{1}{2} \int_0^{\infty} |v(s)|^2 ds + F(W + \int_0^{\cdot} v(s) ds)\right].
\end{align*}
Next, let $F: \clc_0^{\infty} \to \RR$ be a general bounded measurable map.
Let $\{T_n\}_{n\ge 1}$ be an increasing sequence such that $T_n \uparrow \infty$. 
Define $\clg_n \doteq \sigma\{W(s): s \le T_n\}$. Note that by the martingale convergence theorem
$F_n \doteq \EE(F \mid \clg_n)$ converges a.s. to $F$ as $n\to \infty$. Also 
$\|F_n\|_{\infty} \le \|F\|_{\infty}$. Furthermore, since $F_n$ is $\clg_n$ measurable, there is a bounded measurable map $\tilde F_n: \clc_0^{T_n}(\RR) \to \RR$ such that $F_n(W(\cdot)) = 
\tilde F_n(W(\cdot \wedge T_n))$ a.s. This also says that for any $v \in \cla_b$,
$F_n(W(\cdot)+ \int_0^{\cdot} v(s) ds) = \tilde F_n(W(\cdot \wedge T_n)+ \int_0^{\cdot \wedge T_n} v(s) ds)$ a.s.
It thus follows that
$$-\log \EE \exp\{-F_n(W)\} \ge \inf_{v \in \cla_{b}} \EE \left[ \frac{1}{2} \int_0^{\infty} |v(s)|^2 ds + F_n(W + \int_0^{\cdot} v(s) ds)\right]$$
a.s.
Let $v_n \in \cla_b$ be $n^{-1}$-optimal for the right side. Then we have
$$-\log \EE \exp\{-F_n(W)\} \ge  \EE \left[ \frac{1}{2} \int_0^{\infty} |v_n(s)|^2 ds + 
F_n(W + \int_0^{\cdot} v_n(s) ds)\right] -n^{-1}.$$
The left side in the above display converges to $-\log \EE \exp\{-F(W)\}$ as $n\to \infty$.
Denoting by $\theta_n$ (resp. $\theta$) the probability law of $W + \int_0^{\cdot} v_n(s) ds$
(resp. $W$) on $\clc_0^{\infty}$, we have that, for every $n\in \NN$,
$$R(\theta_n \| \theta) \le \frac{1}{2} E\int_0^{\infty} |v_n(s)|^2 ds \le 2\|F\|_{\infty} +1.$$
This shows that (cf. \cite[Lemma 2.5]{buddupbook})
$$\lim_{n\to \infty} \left|F_n(W + \int_0^{\cdot} v_n(s) ds) - F(W + \int_0^{\cdot} v_n(s) ds)\right|= 0.$$
Combining the above
\begin{align*}
	-\log \EE \exp\{-F(W)\} &= \lim_{n\to \infty} -\log \EE \exp\{-F_n(W)\}\\ 
	&\ge \liminf_{n\to \infty}\EE \left[ \frac{1}{2} \int_0^{\infty} |v_n(s)|^2 ds + 
F_n(W + \int_0^{\cdot} v_n(s) ds)\right]\\ 
&= \liminf_{n\to \infty}\EE \left[ \frac{1}{2} \int_0^{\infty} |v_n(s)|^2 ds + 
F(W + \int_0^{\cdot} v_n(s) ds)\right]\\ 
&\ge \inf_{v \cla_b}\EE \left[ \frac{1}{2} \int_0^{\infty} |v(s)|^2 ds + 
F(W + \int_0^{\cdot} v(s) ds)\right].
\end{align*}
This completes the proof of the lower bound and thus the result follows.
\hfill \qed
\begin{remark}
	As in the proof of Theorem \ref{thm:simp}, the infimum in Theorem \ref{thm:infhor} can be taken over $\cla_s$ or $\tilde \cla_s$. Also,  a similar representation can be written for an infinite dimensional Brownian motion, e.g. for a $H$-valued Wiener process as in Section \ref{sec:HS}.
\end{remark}
\section{Functionals of Fractional Brownian Motion}
A representation similar to \eqref{rep:bodu} can be obtained for more general Gaussian processes than a Brownian motion. Fractional Brownian motions are an important class of Gaussian processes with long-range dependence that arise in many applications. In this section we present a variational formula for functionals of a fractional Brownian motion given in \cite{budsong} that has been used for studying large deviation properties of small noise stochastic differental equations with a fractional Brownian motion. The variational formula relies on a general result from \cite{zhang2009variational} on a variational representation for random functionals on abstract Wiener spaces. We begin by presenting this latter result and then describe how a representation for functionals of a fractional Brownian motion can be deduced from it.
\subsection{Abstract Wiener Space Representation.}
For simplicity we consider the time horizon $[0,1]$. The case of a general finite time horizon $[0,T]$ can be treated similarly.
Let $(\WW, \| \cdot\|_{\WW})$ be a separable Banach space. and let $(\HH, \langle \cdot, \cdot\rangle_{\HH}, \|\cdot\|_{\HH})$ be a separable Hilbert space densely and continuously embedded in $\WW$. Let $\mu$ be a centered, unit variance Gaussian measure over $\WW$ (cf. \cite{kuo2006gaussian}). Identifying in the usual manner the dual space $\HH^*$ with itself, $\WW^*$ may be viewed as a dense linear subspace of $\HH$ and we have that for any $\ell \in \WW^*$ and $h \in \HH$, $\ell(h) = \langle \ell, h\rangle_{\HH}$. Denoting the embedding map from $\HH$ into $\WW$ as $i_{\HH}$, the collection $(i_{\HH}, \HH, \WW, \mu)$ is referred to as an abstract Wiener space and $\HH$ is called to Cameron-Martin space of this abstract Wiener space (cf. \cite{G}).

We recall the notion of a filtration on an abstract Wiener space from \cites{ustzak, ustzakbook}.
A collection of projection operators $\{\pi_t, t \in [0,1]\}$ on $\HH$ are referred to as a continuous and strictly monotonic resolution  of the identity in $\HH$ if the following hold 
	\begin{enumerate*}[label=(\roman*)]
  \item $\pi_0 = 0$, $\pi_1 = \mbox{Id}$, where $\mbox{Id}$ is the identity operator, 
  \item for $0 \le s < t \le 1$, $\pi_s \HH$ is a strict subset of $\pi_t\HH$,
  \item for any $h \in \HH$ and $t \in [0,1]$, $\lim_{s\to t} \pi_s h = \pi_t h$.
\end{enumerate*}

By denseness of $\WW^*$ in $\HH$ we can find for every $h \in \HH$ a sequence $\{h_n\} \subset \WW^*$ such that $h_n \to h$. This says that the sequence $\{h_n\}$ regarded as a sequence of random variables on the probability space $(\WW, \clb(\WW), \mu)$ is a Cauchy sequence in $L^2(\mu)$ and thus must converge to a limit $\delta(h)$ in $L^2(\mu)$. This limit is referred to as the Skorohod integral of $h$ and we occasionally write $\delta(h)(w) = \langle h, w\rangle$ for $h \in \HH$ and $w \in \WW$.

Define the filtration $\{\clf_t\}_{0\le t \le 1}$  as $\clf_t \doteq \sigma\{\delta(\pi_th), h \in \HH\} \vee \cln$, where $\cln$ is the collection of all $\mu$-null sets in $\clb(\WW)$. We will regard this as a filtration on the $\mu$-complete probability space  $(\WW, \clf, \mu)$, where $\clf \doteq \clf_1$. Expecteation on this probability space will be denoted as $\EE$.

An $\HH$ valued random variable $v$ on the above probability space is said to be $\clf_t$-adapted if for every $t \in [0,1]$ and $h \in \HH$, $\langle \pi_t h, v\rangle_{\HH}$ is  $\clf_t$ measurable.
The collection of  adapted $\HH$ valued square integrable  random variables (i.e $\EE\|v\|_{\HH}^2 <\infty$) is denoted by $\clh^a$. For $N \in \NN$, denote by 
$S_N$ the ball of radius $N$ in $\HH$ and let 
$\clh^a_{b,N}$ the subcollection of $\clh^a$ consisting of $v$ that take values in $S_N$. Let $\clh^a_b \doteq \cup_{N>0} \clh^a_{b,N}$.
For $t \in [0,1]$, let $\clc_t$ be the collection of cylindrical functions of the form
$$F(w) = g (\langle \pi_t h_1, w\rangle, \cdots , \langle \pi_t h_n, w\rangle), \; g \in C_b^{\infty}(\RR^n), \; h_1, \ldots h_n \in \WW^*.$$
Here $C_b^{\infty}(\RR^n)$ is the space of real infinitely differentiable functions on $\RR^n$, with the function and all its derivatives bounded.
Note that such a $F$ is $\clf_t$ measurable.
A $v \in \clh^a$ is said to be simple if it takes values in $S_N$ for some $N<\infty$ and for some $0= t_0<t_1<\cdots < t_n=1$, $\{h_i\}_{0\le i \le n-1} \subset \HH$, and $\xi_{i} \in \clc_{t_i}$, $i=0, 1, \ldots n-1$,
$$v(w) = \sum_{i=0}^{n-1} \xi_i(w) (\pi_{t_{i+1}}- \pi_{t_i}) h_i.$$
We denote the collection of all such simple $v$ as $\cls^a$.
The following is the main representation from \cite{zhang2009variational}  (see Theorem 3.2 therein).
\begin{theorem}\label{thm:zhang}
	Let $F$ be a real bounded and measurable function on $\WW$. Then
	$$-\log \EE (e^{-F}) = \inf_{v \in \clr} \EE\left (\frac{1}{2} \|v\|^2_{\HH} + F(\cdot + v)\right)$$
	where $\clr \in \{\clh^a, \cls^a\}$.
\end{theorem}
\subsection{Fractional Brownian Motion Representation.}
We will now use the representation in Theorem \ref{thm:zhang}  to provide a more explicit representation for a fractional Brownian motion. We begin by recalling some basic definitions.
For  $H\in(0,1)$, a $d$-dimensional fractional Brownian motion (fBm) $B^H=\{B^H_t:\,t\in[0,1]\}$ with Hurst parameter $H$ defined on some complete probability space $(\Omega, \mathcal{F}, \mathbb{P})$ is a centered Gaussian process whose covariance matrix $R_H=(R_H^{i,j})_{1\leq i,\,j\leq d}$ is given by 
\begin{equation}\label{covariance}
R_H^{i,j}(s, t)=\mathbb{E}(B^{H,i}_sB^{H,j}_t)=\frac{1}{2}(s^{2H}+t^{2H}-|t-s|^{2H})\delta_{i,j},
\; s,\,t\in[0,1],
\end{equation}
 where $\delta$ is the Kronecker delta function.  When $H=\frac{1}{2}$, the above process  is simply a $d$-dimensional standard Brownian motion. 
From the above covariance formula it is immediate  that
\begin{equation*}
\mathbb{E}(\vert B_t^H-B_s^H\vert^2)=d\vert t-s\vert^{2H},\; t,s \in [0,1].
\end{equation*}
From the above property along with Kolmogorov's continuity criterion it follows that the sample paths of $B^H$ are a.s. H\"{o}lder continuous of order $\beta$ for all $\beta<H$.

In what follows $(\Omega,\mathcal{F},\mathbb{P})$ will denote the canonical probability space, where $\Omega=C_0([0,1]:\mathbb{R}^d)$ is the space of continuous functions null at time $0$, equipped with topology of uniform convergence, $\mathcal{F}=\mathcal{B}(C_0([0,1]:\,\mathbb{R}^d))$ is the Borel $\sigma$-algebra and $\mathbb{P}$ is the unique $d$-dimensional probability measure such that the  canonical process $B^H=\{B^H_t(\omega)=\omega(t):\,t\in[0,1]\}$ is a $d$-dimensional fractional Brownian motion with Hurst parameter $H$. Consider the canonical filtration given by $\{\mathcal{F}^H_t:\,t\in[0,1]\}$, where $\mathcal{F}^H_t =\sigma\{B^H_s:\,0\leq s\leq t\}\vee \mathcal{N}$ and $\mathcal{N}$ is the set of the $\mathbb{P}$-negligible events.

 We now recall a representation for a fBM in terms of a standard Brownian motion and a suitable kernel function (cf. \cite{DU}). Let $F(a,b,c;z)$ denote the Gauss hypergeometric function defined for any $a,\,b,\,c,\, z\in\mathbb{C}$ with $|z|<1$ and $c\neq 0,-1,-2,\dots$ by
\[
F(a,b,c;z)=\sum_{k=0}^{\infty}\frac{(a)_k(b)_k}{(c)_k}z^k,
\]
where $(a)_0=1$ and $(a)_k=a(a+1)\dots(a+k-1)$ is the Pochhammer symbol.

Let $B=\{B_t=(B_t^1,\dots,B_t^d),\,t\in[0,1]\}$ be a standard $d$-dimensional Brownian motion.
Then it is well known (cf. \cite{DU} ) that the process
\begin{equation}\label{bm-2-fbm}
B_t^H= \int_0^1K_H(t,s)dB_s, \; t \in [0,1]
\end{equation}
defines a fBm with Hurst parameter $H$, where 
  \begin{equation}\label{K-H}
K_H(t,s)= k_H(t,s) \mathbf{1}_{[0,t]}(s),
\end{equation}
for $0\le s\le t$
\begin{equation}\label{eq:khts}
k_H(t,s) = \frac{c_H}{\Gamma\left(H+\frac{1}{2}\right)}(t-s)^{H-\frac{1}{2}}F\left(H-\frac{1}{2},\frac{1}{2}-H,H+\frac{1}{2};1-\frac{t}{s}\right),
\end{equation}
$c_H=\left[\frac{2H\Gamma\left(\frac{3}{2}-H\right)\Gamma\left(H+\frac{1}{2}\right)}{\Gamma(2-2H)}\right]^{1/2}$,
 and $\Gamma(\cdot)$ is the gamma function. \newline

Define $\mathcal{H}_H=\{(K_H\dot{h}^1,\dots,K_H\dot{h}^d):\,\dot{h}=(\dot{h}^1,\dots,\dot{h}^d)\in L^2([0,1]:\RR^d)\}$, 
where 
$$(K_Hf)(t) \doteq \int_0^1K_H(t,s)f(s)ds \mbox{ for } f \in L^2([0,1]:\RR^d),\; t \in [0,1].$$
That is, any $h\in\mathcal{H}_H$ can be represented as
\[
h(t)=(K_H\dot{h})(t)=\int_0^1K_H(t,s)\dot{h}(s)ds,
\]
for some $\dot{h}\in L^2([0,1]:\RR^d)$. Define a scalar inner product on $\mathcal{H}_H$ by 
\[
\langle h, g\rangle_{\mathcal{H}_H}=\langle K_H\dot{h}, K_H\dot{g}\rangle_{\mathcal{H}_H}=\langle \dot{h}, \dot{g}\rangle_{L^2}.
\]

Then $\mathcal{H}_H$ is a separable Hilbert space with the inner product $\langle \cdot,\cdot\rangle_{\mathcal{H}_H}$. It can be checked (cf. \cite[Remark 1]{budsong}) that $\mathcal{H}_H$ is a subset of $ \Omega=C_0([0,1]:  \mathbb{R}^d)$ and the embedding map is continuous.
We now present the main variational representation for functionals of fBm from \cite{budsong}.
For  $0<N<\infty$, let 
$S_N=\{v\in\mathcal{H}_H: \,\frac{1}{2}\Vert v\Vert^2_{\mathcal{H}_H}\leq N\}$. 

Let $\cla$ be the collection of all measurable $v: \Om \to \clh_H$ such that there is a 
$\clf^H_t$-predictable $\dot{v}: [0,T]\times \Om \to \RR$ such that
$\EE \int_0^1 \|\dot{v}(t)\|^2 dt <\infty$ and $v\om) = K_H \dot{v}(\om)$ a.s. 
The subcollection of $\cla$ consisting of $v$ that take values in $S_N$ will be denoted as $\cla_{b,N}$ and, as before,
$\cla_b = \cup_{N>0}\cla_{b,N}$.
The following is \cite[Proposition 2]{budsong}.
\begin{theorem}\label{representation}
Let $F$ be a real bounded measurable function on $\Omega$. Then we have
\begin{align*}
-\log\mathbb{E}(e^{-F(B^H)})=&\inf_{v\in\clr}\mathbb{E}\left(f(B^H+v)+\frac{1}{2}\Vert v\Vert^2_{\mathcal{H}_H}\right),
\end{align*}
where $\clr \in \{\cla, \cla_b\}$. 
\end{theorem}
\begin{remark}
	Although not discussed here, one can also replace the above infimum by the infimum over the subclass of $\cla_b$ consisting of all simple processes.
\end{remark}
We provide a brief proof sketch of this result. We refer the reader to \cite{budsong} for details. The basic idea is to identify a suitable abstract Wiener space and then apply Theorem \ref{thm:zhang}.

It is easy to verify that functions in $\mathcal{H}_H$ are  $H$-H\"{o}lder continuous (cf. \cite[Lemma 2]{budsong}), in particular we have the following fact:
Any $h\in\mathcal{H}_H$ is in  $C^H([0,1]:\RR^d)$, and $\Vert h\Vert_\infty\leq  \Vert h\Vert_{\mathcal{H}_H}$ and $\Vert h\Vert_{H}\leq  \Vert h\Vert_{\mathcal{H}_H}$. 

Next, recall that $\Omega=C_0([0,1]:\RR^d)$ is a Banach space equipped with the sup-norm $\Vert \cdot \Vert_\infty$.  Let $\Omega^\ast$ be its topological dual. 
We now introduce the abstract Wiener space associated with a fractional Brownian motion.
The following result  is taken from \cite[Theorem 3.3]{DU}.
\begin{lemma}If we  identify  $L^2([0,1]:\RR^d)$ and its dual, we have the following diagram
\[
\Omega^\ast\overset{i_H^\ast}{ \xrightarrow{\hspace{1cm}}} \mathcal{H}_H^\ast\overset{K_H^\ast}{ \xrightarrow{\hspace{1cm}}}L^2([0,1]:\RR^d)\overset{K_H}{ \xrightarrow{\hspace{1cm}}}\mathcal{H}_H\overset{i_H}{ \xrightarrow{\hspace{1cm}}}\Omega
\]
Where 
 $i_H$ is the injection from $\mathcal{H}_H$ into $\Omega$, and $K_H^\ast$ and $i_H^\ast$ are the respective adjoints. 
\begin{itemize}
\item[(a)] The injection $i_H$ embeds $\mathcal{H}_H$ densely into $\Omega$, and $\mathcal{H}_H$ is the Cameron-Martin space of the abstract Wiener space $(i_H,\mathcal{H}_H, \Omega, \PP)$ in the sense of Gross \cite{G}.
\item[(b)] The restriction of $K_H^\ast$ to $\Omega^\ast$ can be represented by 
\[
(K_H^\ast\eta)(s)=\int_0^1K_H(t,s)\eta(dt)=\left(\int_0^1K_H(t,s)\eta_1(dt),\dots,\int_0^1K_H(t,s)\eta_d(dt)\right), 
\]
for any $\eta=(\eta_1,\dots,\eta_d)\in\Omega^\ast$.
\end{itemize}
\end{lemma}

\bigskip
As before, denote for  $h\in\mathcal{H}_H$ its Skorohod integral by $\delta(h)$.
Recall the filtered probability space $(\Omega,\mathcal{F},\mathbb{P},\{\mathcal{F}^H_t\})$
introduced above and recall that 
$\{B^H_t\}$ is the canonical coordinate process
on $(\Omega,\mathcal{F})$.

Define the family $\{\pi_t^H,\, t\in[0,1]\}$ of projection operators in $\mathcal{H}_H$ by
\begin{equation}\label{eq:projfam}
\pi_t^Hh=\pi_t^H(K_H\dot{h})=K_H(\dot{h}\mathbf{1}_{[0,t]}), \; h\in\mathcal{H}_H.
\end{equation}

 

 The following result is a consequence of
  \cite[Proposition 4.4, Theorems 4.3 and 4.8]{DU}.

\begin{lemma}\label{filtration} 
\begin{itemize}
\item[(a)] For any $t\in[0,1]$,  $\mathcal{F}_t^H=\sigma\{\delta(\pi^H_t h),\,h\in\mathcal{H}_H\}\vee \mathcal{N}$. 
\item[(b)] For any  $\mathcal{H}_H$-valued $\{\mathcal{F}^H_t,\,t\in[0,1]\}$-adapted stochastic process $u$ there is a $\{\mathcal{F}^H_t,\,t\in[0,1]\}$-predictable process $\dot{u}: [0,T]\times \Om \to \RR^d$ such that $u = K_H \dot{u}$  a.s.
\end{itemize}
\end{lemma}

One can now complete the proof of Theorem \ref{representation} as follows.
The collection  $\{\pi_t^H,\, t\in[0,1]\}$ introduced in \eqref{eq:projfam} defines 
 a continuous and strictly monotonic resolution 
 of the identity in $\mathcal{H}_H$ in the sense described previously in the section.
 Also, using Lemma \ref{filtration},  the classes $\cla$ and $\cla_b$ in this section are the same as the classes $\clh^a$ and $\clh^a_b$ introduced  earlier in the section, when specialized to the abstract Wiener space for the setting of a fBM. The result is now immediate from Theorem \ref{thm:zhang}.
 \hfill \qed
 
 \section{L\'{e}vy Noise.}
 In this section we present the analogues of the Bou\'{e}-Dupuis variational formulas for L\'{e}vy processes. The results from this section are from \cite{buddupmar2}. They can also be found in \cite[Section 8.2]{buddupbook}. The statement in Theorems \ref{thm:jump} and \ref{thm:jumpdiff} that $\clr$ can be $\cla_s$ or $\tilde \cla_s$, although does not appear in these references, it can be easily deduced from the arguments there along the lines of the proof of Theorem \ref{thm:simp}. We omit the details.
 
 A L\'{e}vy noise is composed of a two random components, one consisting of a finite or an infinite dimensional Brownian motion and the other described through a suitable Poisson random measure(PRM). We begin by first presenting a representation of functionals of a PRM and then we will give the representation for the more general case where the functional depends on both a PRM and an infinite dimensional Brownian motion.
 
 \subsection{Representation for a Poisson Random Measure.}
 Let $\clx$ be a locally compact Polish space and let $\Sigma(\clx)$ denote the space of all locally-finite measures $\nu$ on $(\clx, \clb(\clx))$. The space $\Sigma(\clx)$ is equipped with the standard vague convergence topology. Fix $T \in (0,\infty)$ and let $\clx_T = [0,T]\times \clx$. Let $(\Om, \clf, \PP)$ be  a complete probability space with a filtration
 $\{\clf_t\}_{0\le t \le T}$ satisfying the usual assumptions. Fix $\nu \in \Sigma(\clx)$ and let $\nu_T \doteq \lambda_T \times \nu$ where $\lambda_T$ is the Lebesgue measure on $[0,T]$.

 Let $N$ be a $\clf_t$- Poisson random measure with points in $\clx_T$ and intensity measure $\nu_T$. Recall this means that $N$ is a $\Sigma(\clx_T)$ valued random variable with the following properties
 	\begin{enumerate*}[label=(\roman*)]
   \item for every $t \in [0,T]$ and $A  \in \clb([0,t]\times \clx)$, $N(A)$ is $\clf_t$-measurable,
   \item for every $t \in [0,T]$ and $A  \in \clb((t, T]\times \clx)$, $N(A)$ is independent of $\clf_t$,
   \item for any $k \in \NN$ and $A_1, \ldots A_k \in \clb(\clx_T)$ such that $A_i \cap A_j = \emptyset$ when $i\neq j$, and $\nu_T(\cup_{i=1}^k A_i) <\infty$, we have $N(A_1), \ldots N(A_k)$
   are mutually independent Poisson random variables with parameters $\nu_T(A_1), \ldots \nu_T(A_k)$, respectively.
 \end{enumerate*}

Fix $\theta >0$. Letting $\MM \doteq \Sigma(\clx)$, we will denote by $\PP_{\theta}$ the unique probability measure on $(\MM, \clb(\MM))$ under which the the canonical map $N:\MM \to \MM$ defined as $N(m)=m$ is PRM with intensity $\theta \nu_T$. In typical large deviation problems of interest $\theta$ plays the role of scaling parameter and one is interested in behavior of the system as $\theta$ approaches some limiting value (e.g. $\infty$ or $0$). In such problems one is interested in characterizing the asymptotic behavior of Laplace functionals of the form
$-\log \EE_{\theta} \exp\{-F(N)\}$, where $F: \MM \to \RR$ is some bounded measurable map. 

We will now present the variational representation from \cite{buddupmar2} that, together with weak convergence methods, allows one to study such asymptotic behavior in many problem settings (cf. references in \cite{buddupbook}). In formulating this variational representation it will be convenient to work with a PRM that is defined on a larger point space.

Let $\cly \doteq \clx \times [0, \infty)$ and $\cly_T \doteq [0,T]\times \cly$. Let $\bar \MM$ denote the space of locally finite measures on $\cly_T$. and let $\bar P$ be the unique measure on
$(\bar \MM, \clb(\bar \MM))$ such that the canonical map $\bar N: \bar \MM \to \bar \MM$ defined as $\bar N(m) \doteq m$, $m \in \bar \MM$, is a PRM with points in $\cly_T$ and intensity 
measure $\bar \nu_t \doteq \lambda_T \times \nu \times \lambda_{\infty}$, where $\lambda_{\infty}$ is the Lebesgue measure on $[0,\infty)$.

Let $\clg_t$ denote the augmentation of $\sigma \{\bar N((0,s]\times A): 0 \le s \le t, A \in \clb(\cly)\}$ with all $\bar \PP$ null sets in $\clb(\bar \MM)$ and denote by $\clp\clf$ the $\clg_t$-predictable $\sigma$-field on $[0,T]\times \bar \MM$.

Recall that the representation for functionals of Brownian motions, given e.g. in \eqref{rep:bodu}, involves $L^2$-controls $v$ and controlled noise processes $W + \int_0^{\cdot} v(s) ds$. 
We now introduce the controls and controlled noises that will play the analogous role in the  variational representations for PRM.
First the space of controls, which we denote by $\cla$ will be the collection of all maps $v: [0,T]\times \bar \MM \times \clx$ that are $(\clp\clf \otimes \clb(\clx))\backslash \clb([0,\infty))$ measurable.  For notational ease, for $(t,\om,x) \in [0,T]\times  \bar\MM \times \clx $, we will occasionally suppress $\om$ and write $v(t, \om, x)$ as $v(t,x)$.
 For $v \in \cla$, define a counting process $N^v$ on $\clx_T$ as
$$N^{v}((0,t]\times U) \doteq \int_{(0,t]\times U} \int_{(0,\infty)} \one_{[0, v(s,x)]}(r) \bar N(ds\times dx\times dr)$$
for all $t \in [0,T]$ and $U \in \clb(\clx)$. When, for some $\theta>0$, $v(s,\om, x) =\theta$ for all $(s,\om,x) \in \clx_T \times \bar\MM$, we will write $N^v$ as simply $N^{\theta}$.
The process $N^v$ can be thought of as the controlled noise process associated with the control $v$.  

Next, recall that the Brownian motion representation involved a quadratic cost term given as a suitable $L^2$-norm. In the representation for a PRM the quadratic function is replaced by the following superlinear (but sub-quadratic) function $\ell: [0,\infty) \to [0, \infty)$ defined as
$$\ell(r) \doteq r \log r - r +1, \; r \in [0,\infty).$$
For $v \in \cla$, define
$$L_T(v) \doteq \int_{\clx_T} \ell(v(t,x)) \nu_T(dt\times dx).$$
This is the cost term associated with a control $v$ that will replace the quadratic cost term we see in Brownian motion representation formulas, for the case of a PRM.

In the case of the Brownian motion, it was useful to consider the smaller collection of controls than $\cla$, namely $\cla_b$, which consisted of controls for which one can invoke the Girsanov theorem. We will  consider an analogous smaller collection of controls here as well. Let $\{K_n\}_{n \in \NN}$ be an increasing sequence of compact sets of $\clx$ such that $\cup_{n\ge 1} K_n = \clx$.
For $M \in (0, \infty)$, let
\begin{align*}
	\cla_{b,M} &\doteq \{v \in \cla: L_T(v) \le M, \mbox{ for some } n \in \NN, v(t,x) \in [1/n, n],\\
	&\quad\quad\quad \mbox{ and } v(t,\om,x) = 1\mbox{ if } x \in K_n^c, \mbox{ for all } (t,\om) \in [0,T] \times \bar \MM\}.
\end{align*}
A process in $\cla_{b,M}$ is called simple if the following holds: There exist $n, l, n_1, \ldots n_l \in \NN$; a partition $0=t_0< t_1 < \cdots < t_l=T$; for each $i=1, \ldots , l$ a disjoint
 measurable partition$\{E_{ij}\}_{1\le j \le n_i}$ of $K_n$; $\clg_{t_{i-1}}$- measurable random variables $X_{ij}$, $1\le i\le l$, $1\le j \le n_i$, such that $X_{ij}\in [1/n, n]$; and for
 $(t,\bar m, x) \in [0,T]\times  \Om \times \clx $
$$
v(t, \bar m, x) = \one_{\{0\}}(t) + \sum_{i=1}^l\sum_{j=1}^{n_i} \one_{(t_{i-1}, t_i]}(t) X_{ij}(\bar m) \one_{E_{ij}}(x) + \one_{K_n^c}(x) 1_{(0,T)}(t).$$
We denote the collection all such simple processes as $\cla_{s,M}$.
Define $\cla_b \doteq \cup_{M >0} \cla_{b,M}$ and $\cla_s \doteq \cup_{M >0} \cla_{s,M}$.
As in the Brownian motion case, when the canonical filtration $\{\clg_t\}$ is replaced by a larger filtration $\{\clf_t\}_{0\le t \le T}$ under which
$\bar N$ is a $\clf_t$-PRM with the same intensity (in such a case of course $\bar N$ is defined on a larger probability space than the canonical space $(\bar \MM, \clb(\bar \MM))$), we will denote the classes analogous to $\cla$, $\cla_b$ and $\cla_s$ as $\tilde \cla$, $\tilde \cla_b$, and $\tilde \cla_s$, respectively.
The following result is from \cite{buddupmar2} (see also \cite[Theorem 8.12]{buddupbook}).
\begin{theorem} \label{thm:jump}	
Let $F: \MM \to \RR$ be a bounded measurable map. Then for any $\theta>0$
$$-\log E_{\theta} \exp\{-F(N)\} = \inf_{v \in \clr} \bar \EE[ \theta L_T(v) + G(N^{\theta v})],$$
where $\clr \in \{\cla, \tilde \cla, \cla_b, \tilde \cla_b, \cla_s, \tilde \cla_s\}$.	
\end{theorem}

\subsection{General L\'{e}vy Process Representation.}
Finally, we present the representation for a general L\'{e}vy noise. Let $(\Om, \clf, \PP)$ be a probability space with a filtration $\{\clf_t\}_{0\le t \le T}$ satisfying the usual conditions and assume that this probability space supports the processes introduced below.  With $Q$ as in Section \ref{sec:HS}, let $W$ be a a  $Q$-Wiener process with respect to the filtration $\clf_t$.
Let $\nu, \clx, \clx_T, \cly_T, \bar \nu_T$ be as introduced above. Let $\bar N$ be a $\clf_t$-PRM with points in $\cly_T$ with intensity measure $\nu_T$. Assume that for all $0\le s \le t< \infty$, $(\bar N((s,t]\times \cdot)), W(t) -W(s))$ is independent of $\clf_s$. Let $\clp\clf$ be the $\clf_t$-predictable $\sigma$-field on $[0,T]\times \Om$. Denote by $\tilde \cla^W$ and $\tilde \cla^W_b$ the collections $\tilde \cla$ and $\tilde \cla_b$ introduced below Defintiion \ref{def:hbm}. Also, denote by $\tilde \cla^N$ and $\tilde \cla^N_b$ the collections $\tilde \cla$ and $\tilde \cla_b$ introduced earlier in the current section. Let $\tilde \cla_b \doteq \tilde \cla^W_b \times \tilde \cla^N_b$ and $\tilde \cla \doteq \tilde \cla^W \times \tilde \cla^N$.
For $v=(\psi, \varphi) \in \tilde \cla$, let $L_T^W(\psi) \doteq \frac{1}{2} \int_0^T \|\psi(s)\|_0^2 ds$, where the norm $\|\cdot\|_0$ is as introduced in Section \ref{sec:HS},
and let $L_T^N(\varphi) \doteq  \int_{\clx_T} \ell(\varphi(t,x)) \nu_T(dt\times dx)$. Also, set $ L_T(v) \doteq L_T^W(\psi) + L_T^N(\varphi)$.
Next, for $\psi \in \tilde \cla^W$, let $W^{\psi}(t) \doteq W(t) + \int_0^t \psi(s) ds$, $t\in [0,T]$.

The following is \cite[Theorem 8.19]{buddupbook}.

\begin{theorem}\label{thm:jumpdiff}
	Let $F: \clc_0^T(H) \times \MM \to \RR$ be a bounded measurable map. Then for $\theta \in (0,\infty)$
	$$-\log \EE \exp\{- F(W, N^{\theta})\} = \inf_{v = (\psi, \varphi) \in \clr} E\left[\theta L^T(v) + F(W^{\sqrt{\theta}}, N^{\theta \psi})\right],$$
	where $\clr \in \{\tilde \cla_b, \tilde \cla\}$. 
\end{theorem}

\begin{remark}
	In a similar manner as in Section \ref{sec:extran} one can give a variational representation for functionals that, in addition to depending on a Brownian motion and a PRM, also depend on another independent random variable. Also, in a manner similar to Section \ref{sec:infpath} one can obtain a variational representation for functionals that depend on an infinite path of a Brownian motion and a PRM, namely a function of $\{(W(t), N^{\theta}((0,,t]\times \cdot)), 0\le t<\infty\}$. We omit the details.
\end{remark}

\vspace{0.7cm}

\noindent {\bf Acknowledgements}
Research supported in part by the NSF (DMS-2134107 and DMS-2152577). 



\bibliographystyle{plain}

\bibliography{main2}

\vspace{\baselineskip}
%
%
\scriptsize{
\textsc{\noindent A. Budhiraja \newline
Department of Statistics and Operations Research, and\newline
School of Data Science and Society\newline
University of North Carolina\newline
Chapel Hill, NC 27599, USA\newline
email:  budhiraj@email.unc.edu \vspace{\baselineskip} }

}

\end{document}